\documentclass{article}

\usepackage[english]{babel}
\usepackage{amssymb}
\usepackage[all]{xy}

\newtheorem{definition}{Definition}[section]

\newtheorem{lemma}[definition]{Lemma}
\newtheorem{corollary}[definition]{Corollary}
\newtheorem{proposition}[definition]{Proposition}

\newtheorem{remark}[definition]{Remark}

\newenvironment{proof}{\vspace{3pt}\textsc{Proof:}\quad }
                       {\hfill \hbox{q.e.d.}\vspace{3pt}}

\def\A{{\mathcal A}}

\def\J{{\mathcal J}}
\def\O{{\mathcal O}}
\def\P{{\mathcal P}}

\def\LL{{\mathcal L}}

\def\AA{\mathbf{A}}
\def\JJ{\mathbf{J}}
\def\LL{\mathbf{L}}
\def\RR{\mathbf{R}}

\def\sub{\subseteq }

\renewcommand{\land}{\mathrel\&}
\def\r{\mathrel r}

\def\overlap{\, \between \,}

\newcommand{\fusim}{\cdot\mkern-4mu | \mkern-4mu\cdot}

\def\comp{\succ}

\title{A constructive Galois connection\\ between closure and interior}
\author{Francesco Ciraulo\footnote{Dipartimento di Matematica, Universit\`a di Padova,
 Via Trieste, 63 - I-35121 Padova, Italy,
\texttt{\{ciraulo,sambin\}@math.unipd.it}.} \and Giovanni
Sambin$^*$}
\date{}

\begin{document}
\maketitle

\begin{abstract}
We construct a Galois connection between closure and interior
operators on a given set. All arguments are intuitionistically valid. Our construction is an intuitionistic version of the classical correspondence between closure and interior operators via complement.
\end{abstract}


\bigskip

In classical mathematics,  the theory of closure operators and that of interior operators 
can be derived one from another. 
In fact,   $\A$ is a closure operator if and only if its companion $-\A -$ (where $-$ is complementation)  is an interior operator.
Since passing to the companion is an involution, one derives that $\J$ is an interior operator if and only if $-\J -$ is a closure operator.

>From an intuitionistic point of view, the picture is more complex. In fact, $-\A -$ is not in general an interior operator.
So the notion of companion has to be defined differently.
Our proposal is based on the notion of \emph{compatibility} between two operators on subsets of a given set. 
We show intuitionistically that every closure operator $\A$ has a greatest compatible interior operator $\JJ(\A)$.
Since classically  $\JJ(\A) = - \A -$, we choose  $\JJ(\A)$ as the companion of $\A$.
Dually, the companion of an interior operator $\J$ is the greatest closure operator $\AA(\J)$ which is compatible with $\J$. 
Classically $\AA(\J) = - \J -$. 

We prove that $\AA$ and $\JJ$ form a Galois connection between closure and interior operators on given set, that is 
$\A \sub \AA(\J)$ if and only if  $\J \sub \JJ(\A)$. Classically, this collapses to the triviality $\A\sub-\J-$ if and only if $\J\sub-\A-$.

In section 1, we start by analysing the notion of compatibility between arbitrary operators  on the same set. 
We specialise to the case of compatibility between a closure and an interior operator  in section \ref{section Galois}.
There we present the constructions of $\AA$ and $\JJ$
and  prove that they form a Galois connection.

Following \cite{bp}, a set equipped with both a closure and an interior operator which are compatible is called a \emph{basic topology}.
In section~\ref{section basic topologies}, we introduce
  two classes of basic topologies:
\emph{saturated} basic topologies, in which the reduction is completely determined by the saturation, and \emph{reduced} ones,
symmetrically. 
We show that the Galois connection can be seen as the composition of two adjunctions between these two classes and all basic topologies.

Classically, saturated and reduced basic topologies coincide. This is not the case intuitionistically as it is shown by a couple of counterexamples we give in section~\ref{counterexamples}. Indeed, we show that each of the two  inclusions between these classes is equivalent to the law of excluded middle.

The constructions we propose are impredicative. However, they can be carried on
predicatively in many important cases. This topic is developed in section~\ref{section predicative}.

\section{Operators on subsets and compatibility}

This section contains some preliminaries about operators on subsets.
The relation of \emph{compatibility} between operators, introduced in \cite{bp}, is recalled and its basic properties are studied. Compatibility
between closure and interior operators will play a fundamental role
in the following sections.

We write $ Pow(S)$ for the collection of all subsets
of a set $S$. An \emph{operator on (the subsets of)} $S$ is a map $\O: Pow(S)\rightarrow Pow(S)$. For future reference, we
fix notation for some operators on a set $S$:
\begin{eqnarray}\label{eq. def. trivial operators}
id & \stackrel{def}{=} & \textrm{the identity map on }  Pow(S),\nonumber \\
- & \stackrel{def}{=} & \textrm{the intuitionistic pseudo-complement},\nonumber \\
const_U & \stackrel{def}{=} & \textrm{the operator with
constant value } U\sub S, \\
\bot & \stackrel{def}{=} & const_\emptyset, \nonumber\\
\top & \stackrel{def}{=} & const_S. \nonumber
\end{eqnarray}
For two operators $\O_1$, $\O_2$ on the same set
$S$, we write $\O_1\sub\O_2$ if $\O_1(U)\sub\O_2(U)$ for all $U\sub
S$. This is clearly a partial order on the collection of all
operators on $S$. This poset is a complete lattice and, actually, a
frame (for this and other order-theoretic notions see \cite{stone
spaces}, \cite{joyal-tierney}). For every family $\{\O_i$ $|$ $i\in
I\}$ of operators on $S$, its join $\bigvee_{i\in I}\O_i$
and its meet $\bigwedge_{i\in I}\O_i$ are constructed ``pointwise'':
\begin{equation}\label{eq def inFi and supFi}
(\bigvee_{i\in I}\O_i)W\stackrel{def}{=}\bigcup_{i\in I}(\O_i
W)\quad\textrm{and}\quad(\bigwedge_{i\in
I}\O_i)W\stackrel{def}{=}\bigcap_{i\in I}(\O_i W)
\end{equation} for all $W\sub S$.
An operator $\O$ is \emph{monotone} (or \emph{order-preserving}) if
$\O(U)\sub\O(V)$ whenever $U\sub V$; it is \emph{idempotent} if
$\O\O=\O$ (we use juxtaposition for composition). All operators of
equation (\ref{eq. def. trivial operators}), except the
pseudo-complement $-$, are monotone and idempotent. The two
operators of equation (\ref{eq def inFi and supFi}) are monotone if
so is each $\O_i$; thus monotone operators on a set form a subframe
of the frame of all operators on that set.\footnote{On the contrary, each
$\O_i$ being idempotent (or even monotone and idempotent) 
forces neither $\bigwedge_{i\in I}\O_i$ nor $\bigvee_{i\in I}\O_i$ to be
idempotent too. Here are two counterexamples. Assume $S$ is
equipped with a (non discrete) topology and let $int$ and $cl$ be the topological interior and, respectively, closure operators on it. If $U\sub S$ is not
open, then $int\wedge const_U$ is not idempotent (apply it to $S$).
Similarly, $cl\vee const_V$ is not idempotent (apply it to
$\emptyset$), provided that $V$ is not closed.}

We write
$Fix(\O)$ for the collection of all fixed points of the operator $\O$. Note that, provided that $\O$ is idempotent,
the elements of $Fix(\O)$ are all and only the subsets of the form
$\O(W)$ for some $W\sub S$.

\subsection{Compatibility}

In doing mathematics intuitionistically, we need to
distinguish inhabited subsets from merely non-empty ones. To this
aim, as in \cite{some points,bp}, we adopt the symbol $\overlap$ of overlap to express \emph{inhabited} intersection between
two subsets:
\begin{equation}\label{eq. def. overlap}U\overlap
V\quad\stackrel{def}{\Longleftrightarrow}\quad(\exists a\in S)(a\in
U\cap V)\end{equation}for $U,V\sub S$. So $U\overlap V$ is
intuitionistically stronger than, though classical equivalent to,
$U\cap V\neq\emptyset$. Overlap allows us to express in a simple way the
following relation between two operators.

\begin{definition}\label{def. compatibility}
Let $\O$ and $\O'$ be two operators on the same set $S$. We say that
$\O$ is \emph{(left-)compatible} with $\O'$ (and that $\O'$ is
\emph{(right-)compatible} with $\O$) and we write $\O\comp\O'$ if
\begin{equation}
\O U\overlap\O' V\quad\Longrightarrow\quad U\overlap\O' V
\end{equation} for all $U,V\sub S$.
\end{definition} If $cl$ and $int$ are the closure and interior
operators on a topological space $S$, then $cl\comp int$
holds.\footnote{Actually, also $int\comp cl$ holds; however this is
of little interest since $int$ is left-compatible with all
operators, as it happens to every operator contained in the
identity.} In fact, if a point $a$ is in the closure of a set $U$,
then every open neighborhood of $a$ must ``overlap'' $U$. This
argument is valid also intuitionistically as far as one defines $cl
U$ as the set of adherent points of $U$. The motivation for studying
the relation $\comp$ lies in the fact that it captures much of what
is intuitionistically valid about the link between $cl$ and $int$
(see also section~\ref{Compat  saturations  reductions}).
In this section we prove some properties of $\comp$ in the case of
arbitrary operators. The study of compatibility between closure and
interior operators will be recovered in the following section.

Since $\neg(U\overlap V)$ is equivalent to $U\cap V=\emptyset$, the
definition of $\O\comp\O'$ entails $U\cap\O' V=\emptyset\Rightarrow\O U\cap\O' V=\emptyset$ for all
$U$ and $V$. The converse holds classically, but not intuitionistically. To see this, consider the operators $--$ and $\top$. Then $U\cap\top V=\emptyset$ $\Rightarrow$ $--U\cap\top V=\emptyset$ holds, while $--U\overlap\top V$ $\Rightarrow$ $U\overlap\top V$ is tantamount to the logical formula $\neg\neg\exists
x\varphi\rightarrow\exists x \varphi$.


Assuming  $\O$ to be monotone, one can prove that  $U\cap\O' V=\emptyset\Rightarrow\O U\cap\O' V=\emptyset$ for all
$U,V$ is equivalent to
$\O-\O'\sub-\O'$ and also to
$\O'\sub(-\O-)\O'$.

By an easy verification, the following hold for every set $S$:
\begin{equation}\label{eq es. comp. fra operatori}
\begin{array}{r@{\;\comp\;}l@{\qquad}l} id & \O & \textrm{for every operator }\O;\\
\O & id & \textrm{if and only if }\O\sub id;\\
const_U & \O & \textrm{if and only if }\O\sub -const_U(=const_{-U});\\
\O & \bot & \textrm{for every operator }\O;\\
- & \O & \textrm{if and only if }\O=\bot.
\end{array}
\end{equation}

Classically, one also has
$\O\comp-\;$  if and
only if $\;\O\sub id$.

\begin{lemma}\label{lemma comp. and unions}
Let $\O$ and $\{\O_i$ $|$ $i\in I\}$ (for $I$ a set) be operators on
a set $S$. The following hold:
\begin{enumerate}
\item if $\O\comp\O_i$ for
every $i\in I$, then $\O\comp(\bigvee_{i\in I}\O_i)$;
\item if $\O_i\comp\O$ for
every $i\in I$, then $(\bigvee_{i\in I}\O_i)\comp\O$.
\end{enumerate}
\end{lemma}
\begin{proof}
1. Assume $\O U\overlap(\bigvee_{i\in I}\O_i) V$ $=$ $\bigcup_{i\in
I}(\O_i V)$. Then there exists $i\in I$ such that $\O U\overlap\O_i
V$. Since $\O\comp\O_i$, one has $U\overlap\O_i V$. A fortiori
$U\overlap\bigcup_{i\in I}\O_i V$. 2. If $(\bigvee_{i\in I}\O_i)
U\overlap\O V$, then there exists $i\in I$ such that $\O_i
U\overlap\O V$. Since $\O_i\comp\O$, one has $U\overlap\O V$, as
wished.
\end{proof}

The analogous statement for intersections does not hold. As for the
analogous of $1$, consider the following counterexample in a
classical setting (use the classically-valid characterization of compatibility): given the reals with their natural topology, one
has both $(int)(cl)\comp const_{(-\infty,0]}$ and $(int)(cl)\comp
const_{[0,+\infty)}$, but not $(int)(cl)\comp
const_{(-\infty,0]}\cap const_{[0,+\infty)}=const_{\{0\}}$. The
analogous of $2$ holds for $I$ inhabited (this follows from item
$1$ in the following lemma), but not for the empty intersection
$\top$ (in fact, $\top\comp\O$ only if $\O=\bot$).

\begin{lemma}\label{lemma Trentinaglia esteso}
Let $\O$, $\O'$ and $\O''$ be operators on a set $S$; then the
following hold:
\begin{enumerate}
\item if $\O''\sub\O$ and $\O\comp\O'$,
then $\O''\comp\O'$;
\item if $\O\comp\O'$ and $\O''\comp\O'$, then $\O\O''\comp\O'$;
\item if $\O\comp\O'$, then $\O\comp\O'\O''$.
\end{enumerate}
\end{lemma}
\begin{proof}
$1.$ If $\O'' U\overlap\O' V$, then $\O U\overlap\O' V$ (because
$\O''\sub\O$), hence $U\overlap\O' V$ (because $\O\comp\O'$). $2.$
If $\O\O'' U\overlap\O'V$, then $\O'' U\overlap\O'V$ (because
$\O\comp\O'$) and hence $U\overlap\O'V$ (because $\O''\comp\O'$).
$3.$ If $\O U\overlap\O'\O'' V$, then $U\overlap\O'\O'' V$ (because
$\O\comp\O'$).
\end{proof}

\subsubsection{On the greatest  compatible operators}

For every $\O$,  the operator $\bot$ is both the least operator which is
left-compatible with $\O$ and the least operator which is right-compatible
with $\O$. Now a natural question is whether the greatest
left-compatible and the greatest right-compatible operators exist as well.
We can easily show (by means of an intuitionistic, though impredicative proof) that the answer is affirmative.

\begin{proposition}\label{prop. existence of LL(O) and RR(O)} 
For every operator  $\O$ on a set $S$,  both the greatest
operator left-compatible and the greatest operator right-compatible with $\O$ exist and
 are denoted
by $\LL(\O)$ and $\RR(\O)$, respectively.
\end{proposition}

\begin{proof}
Put $\LL(\O)$ $\stackrel{def}{=}$ $\bigvee\{\O'$ $|$
$\O'\comp\O\}$ and $\RR(\O)$ $\stackrel{def}{=}$ $\bigvee\{\O'$ $|$
$\O\comp\O'\}$; then apply lemma \ref{lemma comp. and unions}.
\end{proof}

As a first stock of examples, the properties displayed in
(\ref{eq es. comp. fra operatori}) give:
\begin{equation}\label{eq. LL and RR casi banali}
\RR(id)=\top,\ \LL(id)=id,\ \RR(const_U)=const_{-U},\
\LL(\bot)=\top,\ \RR(-)=\bot.
\end{equation}

By the very definition of $\LL(\O)$ and item $1$ of
lemma \ref{lemma Trentinaglia esteso} it follows that:
\begin{equation}\label{eq. semi-Galois connection}
\O'\comp\O\qquad\textrm{if and only if}\qquad\O'\sub\LL(\O)\,
.\end{equation}
The rest of this section is devoted to finding a more explicit
characterization of $\LL(\O)$ and $\RR(\O)$. We start with the
former.

\begin{proposition}
The operator $\LL(\O)$ satisfies:
\begin{equation}\label{eq def L(O)}
a\in\LL(\O)\,U\quad\Longleftrightarrow\quad(\forall\, V\sub
S)(a\in\O V\ \Rightarrow\ U\overlap\O V)\end{equation} for all $a\in
S$ and $U\sub S$.
\end{proposition}
\begin{proof}
Let $\LL'(\O)$ be the operator defined by the right-hand side of (\ref{eq def L(O)}),
that is,  $a\in\LL'(\O)\,U\stackrel{def}{\Longleftrightarrow}(\forall\, V\sub
S)(a\in\O V\ \Rightarrow\ U\overlap\O V)$.
Then for every operator $\O'$,  compatibility  $\O' U$ $\overlap$
$\O V$ $\Rightarrow$ $U$ $\overlap \O V$ (for all $U,V\sub S$) can
be rewritten as $a\in\O' U$ $\land$ $a\in\O V$ $\Rightarrow$
$U\overlap\O V$ (for all $a\in S$ and $U,V\sub S$), that is, $a\in\O'
U$ $\Rightarrow$ $(\forall V\sub S)(a\in\O V$ $\Rightarrow$
$U\overlap\O V)$ (for all $a\in S$ and $U\sub S$).
So $\O' \comp \O$ if and only if $\O' \sub \LL'(\O)$. This shows that
$ \LL'(\O)$ is the greatest operator left-compatible with $\O$ and hence
$ \LL(\O) =  \LL'(\O)$, that is the claim.
\end{proof}

In order to reach a more explicit description of $\RR(\O)$, we start
with:

\begin{definition}\label{def. splitting subset}
For every operator $\O$ on a set $S$, we say that a subset $Z\sub S$ \emph{splits} $\O$ if $\O U\overlap Z$ $\Rightarrow$
$U\overlap Z$ for all $U\sub S$.
\end{definition}
In other words, $Z$ splits $\O$ if and only if $\O\comp const_Z$.
Conversely, note  that
$\O\comp\O'$ if and only if $\O' V$ splits $\O$ for all $V\sub S$.

\begin{lemma}
Let $\O$ be an operator on a set $S$. The collection of all subsets
that split $\O$ is a sub-suplattice of $Pow(S)$.
\end{lemma}
\begin{proof}
We show that the union of splitting subsets is splitting too. If $Z_i$ splits
$\O$ for all $i\in I$, then $\O\comp const_{Z_i}$ for all $i\in I$,
hence (lemma \ref{lemma comp. and unions}) $\O\comp(\bigvee_{i\in
I}const_{Z_i})$ = $const_{\bigcup_{i\in I}Z_i}$; so $\bigcup_{i\in
I}Z_i$ splits $\O$.
\end{proof}

As a corollary, one gets that $\bigcup\{Z\sub S\ |\ Z\textrm{ splits
}\O\}$ is the largest subset of $S$ that splits $\O$.

\begin{proposition}
The operator $\RR(\O)$ is the constant operator with value the
largest subset that splits $\O$.
\end{proposition}
\begin{proof}
The constant operator with value the largest subset that splits $\O$
is $const_{\bigcup\{Z\ |\ Z\textrm{ splits }\O\}}$. This can be
rewritten as $\bigvee\{const_Z\ |\ Z\textrm{ splits }\O\}$, that is,
$\bigvee\{const_Z\ |\ \O\comp const_Z\}$. By the construction of
$\RR(\O)$ in proposition \ref{prop. existence of LL(O) and RR(O)},
only the inclusion $\bigvee\{\O'\ |\ \O\comp\O'\}$ $\sub$
$\bigvee\{const_Z\ |\ \O\comp const_Z\}$ needs to be checked. So,
let $\O'$ be such that $\O\comp\O'$; we must prove that
$\O'\sub\bigvee\{const_Z\ |\ \O\comp const_Z\}$.  The hypothesis $\O\comp\O'$ means that $\O'V$ splits $\O$ for all $V\sub
S$, that is, $\O\comp const_{\O'V}$ for all $V\sub S$. Therefore, it is sufficient to check that
$\O'\sub\bigvee\{const_{\O' V}\ |\ V\sub S\}$, which is trivial.
\end{proof}

Note that $\O\comp\O'$ implies $\O'\sub\RR(\O)$ because $\RR(\O)$ is the greatest operator which is right-compatible with $\O$. The converse fails,
in general (however, see (\ref{eq. J-semi-Galois}) in proposition \ref{caratterizzazione di J(A)}); here is a counterexample. Let $\O$ be the operator on
$S$ defined by $\O(U)$ = $\{a\in S$ $|$ $U$ is inhabited$\}$
(classically, $\O(\emptyset)=\emptyset$ and $\O(U)=S$ otherwise). It
is easy to check that $S$ splits $\O$; hence $\RR(\O)$ = $const_S$ =
$\top$. If $S$ contains at least two distinct elements $a$ and $b$
say, then $const_{\{b\}}\sub\RR(\O)$ but $\O\not\comp const_{\{b\}}$
because $\O(\{a\})\overlap\{b\}$ $\nRightarrow$
$\{a\}\overlap\{b\}$.

\section{A Galois connection between saturations and reductions}\label{section Galois}

>From now on, we restrict our attention to closure
and interior operators. For the sake of greater generality, we adopt the
following

\begin{definition}
Let $\A$ and $\J$ be two monotone and idempotent operators on $S$.
We say that:
\begin{itemize}
\item[] $\A$ is a \emph{saturation}, or \emph{(generalized) closure
operator}, if it is \emph{expansive}, that is, $id\sub\A$;
\item[] $\J$ is a \emph{reduction}, or \emph{(generalized) interior
operator}, if it is \emph{contractive}, that is, $\J\sub id$.
\footnote{The definitions of saturation and reduction make sense also when $(Pow(S),\sub)$ is replaced with an arbitrary partial order. However, to be able to express the
notion of compatibility one needs some extra structure, as in the notion of \emph{overlap algebra} introduced in \cite{bp}. Almost all definitions and results in this paper
can be restated in a natural way in that framework. See
\cite{cisa} for some of the basic facts.}
\end{itemize}
\end{definition}

Of course, the topological operators of closure and of interior  are examples
of saturations and reductions, respectively. However, saturations
and reductions are more general notions since they usually lack some
standard topological properties such as $\J(U\cap V)=\J U\cap\J V$
or $\A\emptyset=\emptyset$. Among the operators of equation
(\ref{eq. def. trivial operators}), $\top$ is a saturation, $\bot$ a
reduction and $id$ is both a saturation and a reduction; also the
double negation operator $--$ is a saturation. With no exceptions,
in this paper $\A$ (also with subscripts) will always stand for a
saturation, while $\J$ for a reduction.

A general method for constructing saturations and reductions is
well-known.
For any family $\P\sub Pow (S)$, let
\begin{eqnarray}\label{eq def A and J from a family}
\A_\P(U) & \stackrel{def}{=} & \bigcap\{V\in\P\ |\ U\sub V\} \\
\J_\P(U) & \stackrel{def}{=} & \bigcup\{V\in\P\ |\ V\sub U\}
\nonumber\end{eqnarray} for all $U\sub S$. It is straightforward
to check that $\A_\P$ and $\J_\P$ are a saturation and a reduction
on $S$, respectively. Moreover, every saturation and every reduction
can be obtained in this way, namely: $\A$ $=$ $\A_{Fix(\A)}$ for
every saturation $\A$ and $\J$ $=$ $\J_{Fix(\J)}$ for every
reduction $\J$. More precisely, it can be shown that $\A_\P$ is the
least saturation which fixes $\P$ pointwise (that is, $\P\sub
Fix(\A_\P)$ and if $\P\sub Fix(\A)$, then $\A_\P\sub\A$) and that
$\J_P$ is the greatest reduction fixing $\P$.

\begin{definition}\label{def SAT(S) and RED(S)}
For every set $S$, we write $SAT(S)$ for the collection of all saturations on
$S$ and $RED(S)$ for the collection of all reductions on $S$.
\end{definition}
It is routine to prove that:
\begin{proposition}\label{prop infAi and supJi} For every set $S$,
we have:
\begin{enumerate}
\item for every family $\{\A_i\}_{i\in I}$ of saturations on $S$,
the operator $\bigwedge_{i\in I}\A_i$ is a saturation too and hence
$SAT(S)$ is a sub-inflattice of the collection of all operators on
$S$;
\item for every family $\{\J_i\}_{i\in I}$ of reductions on $S$,
the operator $\bigvee_{i\in I}\J_i$ is a reduction too and hence
$RED(S)$ is a sub-suplattice of the collection of all operators on
$S$.
\end{enumerate}
\end{proposition}

The identity operator $id$ is both the bottom of $SAT(S)$
and the top of $RED(S)$. The top in $SAT(S)$ is the
operator $\top$; symmetrically, the bottom in $RED(S)$ is the
operator $\bot$.

The following lemma will be used several times in this paper.

\begin{lemma}\label{lemma A1subA2 J1subJ2} For all $\A_1,\A_2\in
SAT(S)$ and all $\J_1,\J_2\in RED(S)$, the following hold:
\begin{enumerate}
\item $\A_1\sub\A_2$ $\Longleftrightarrow$ $\A_2\A_1=\A_2$
$\Longleftrightarrow$ $\A_1\A_2=\A_2$ $\Longleftrightarrow$
$Fix(\A_2)\sub Fix(\A_1)$;
\item $\J_1\sub\J_2$ $\Longleftrightarrow$ $\J_1\J_2=\J_1$
$\Longleftrightarrow$ $\J_2\J_1=\J_1$ $\Longleftrightarrow$
$Fix(\J_1)\sub Fix(\J_2)$.
\end{enumerate}
\end{lemma}
\begin{proof} 1. If $\A_1\sub\A_2$, then $\A_2\A_1\sub\A_2$
because $\A_2$ is monotone and idempotent; also $\A_2\sub\A_2\A_1$
since $\A_1$ is expansive and $\A_2$ is monotone and so
$\A_2\A_1=\A_2$. If $\A_2\A_1=\A_2$, then
$\A_1\A_2\sub\A_2\A_1\A_2=\A_2$ because $\A_2$ is expansive and
idempotent; also $\A_2\sub\A_1\A_2$ because $\A_1$ is expansive and so
$\A_1\A_2=\A_2$. If $\A_1\A_2=\A_2$, then $\A_1(\A_2 U)=\A_2 U$ for
all $U\sub S$, that is, $Fix(\A_2)\sub Fix(\A_1)$.
Assume $Fix(\A_2)\sub Fix(\A_1)$. For all $U\sub S$ one has $\A_1 U\sub\A_1\A_2
U$ since $\A_2$ is expansive and  $\A_1$ is monotone;
hence $\A_1 U \sub \A_2 U$ because $\A_2 U$  is fixed by $\A_2$ and so,
by assumption, it is fixed also by $\A_1$.
2. Similarly.
\end{proof}

%

Note that, for every $\A_1$, $\A_2$ in $SAT(S)$, the composition
$\A_1\A_2$ need not be a saturation (since it can fail to be idempotent).
Actually, by using lemma \ref{lemma A1subA2 J1subJ2} one can prove that both $\A_1\A_2$
and $\A_2\A_1$ are in $SAT(S)$ if and only if $\A_1\A_2$ $=$
$\A_2\A_1$.
A similar remark holds for reductions.

\subsection{Compatibility between saturations and reductions}
\label{Compat  saturations  reductions}

Classically, the closure and the interior operators of a topological
space are linked one another by the equations $cl$ $=$ $-int-$ and
$int$ $=$ $-cl-$, where $-$ is classical complement. Thus each of
the two operators can be defined by means of the other. These facts
are generally not true from an intuitionistic point of view. On the
other hand, the relation $\comp$ of compatibility provides a more
general link between closure and interior operators. In fact,
 as we saw after definition~\ref{def. compatibility},
$cl\comp int$ is intuitionistically valid. 
Classically, $cl\comp int$ is equivalent both to
$cl\sub -int-$ and to 
$int\sub -cl-$;\footnote{Here is a proof. First, one can rewrite compatibility 
as $U\sub - int V \Rightarrow cl U \sub - int V$. For $U = -V$, 
since $-V \sub - int V$, one gets $cl -V \sub - int V$ for all $V$.
Hence $int \sub - cl -$, which is equivalent to $cl \sub - int -$.
Conversely, let $U\sub - int V$. Then, by applying $- int -$, 
also $- int - U \sub - int V$ and hence $cl U \sub - int V$ by the assumption $cl \sub - int -$.}
 so it expresses ``half'' of
the usual requirement. Actually, since $int\sub -cl-$ is equivalent
to $Fix(int)\sub Fix(-cl-)$ (lemma \ref{lemma A1subA2 J1subJ2}), the condition $cl\comp int$ says
precisely that the topology corresponding to $int$ is coarser (has
fewer opens sets) than the topology corresponding to $cl$. For
instance, if $cl$ is the closure operator for the natural topology
on the reals $\mathbb{R}^2$ and  $int_Z$ is the
interior corresponding to the Zariski topology, then $cl\comp
int_Z$.

When $\comp$ is restricted to a relation between $SAT(S)$ and
$RED(S)$, the examples in (\ref{eq es. comp. fra operatori}) give:

\begin{equation}\label{eq. compatibility casi banali}
\begin{array}{r@{\;\comp\;}l@{\qquad}l} id & \J & \textrm{for every reduction }\J;\\
\A & id & \textrm{if and only if }\A=id;\\
\top & \J & \textrm{if and only if }\J=\bot;\\
\A & \bot & \textrm{for every saturation }\A.
\end{array}
\end{equation}
For a given $\P\sub Pow(S)$, the operators $\A_\P$ and
$\J_\P$ of equations (\ref{eq def A and J from a family}) are in
general \emph{not} compatible.
In fact, let $W$ be an inhabited
subset of a set $S$ and consider the singleton family $\P$ $=$
$\{W\}$. Then $\A_\P\emptyset$ $=$ $W$ $=$ $\J_\P S$. So
$\A_\P\emptyset\overlap\J_\P S$ holds but $\emptyset\overlap\J_\P S$
does not.

\begin{lemma}\label{lemma Trentinaglia}
Let $S$ be a set, $\A\in SAT(S)$ and $\J\in RED(\J)$. If
$\A\comp\J$, then:
\begin{enumerate}
\item $\A'\comp\J$ for all $\A'\in SAT(S)$ such that $\A'\sub\A$;
\item $\A\comp\J'$ for all $\J'\in RED(S)$ such that $\J'\sub\J$.
\end{enumerate}
\end{lemma}
\begin{proof}
Item $1$ is just $1$ of lemma \ref{lemma Trentinaglia esteso}.
Item $2$ follows from 3 of lemma \ref{lemma Trentinaglia esteso} and lemma~\ref{lemma A1subA2 J1subJ2}.
\end{proof}

By (\ref{eq. compatibility casi banali}),   $\bot$
is the least reduction compatible with a given saturation $\A$ and
$id$ is the least saturation compatible with a given reduction $\J$.
We now face the dual problem: to find the \emph{greatest saturation}
compatible with a given $\J$ and the \emph{greatest reduction}
compatible with a given $\A$.

\begin{definition}\label{def A(J) and J(A)}
Let $\J\in RED(S)$; when it exists, the greatest saturation
compatible with $\J$ is denoted by $\AA(\J)$. Similarly, when the
greatest reduction compatible with a given $\A\in SAT(S)$ exists, it
is denoted by $\JJ(\A)$.
\end{definition}

The facts in (\ref{eq. compatibility casi banali}) show that
$\AA(id)$, $\AA(\bot)$, $\JJ(id)$, $\JJ(\top)$ all exist and one
has:
\begin{equation}\label{eq. AA casi banali}\label{eq. JJ casi banali}
\AA(id)=id \textrm{ ,}\quad \AA(\bot)=\top \textrm{ ,}\quad
\JJ(id)=id \textrm{ ,}\quad \JJ(\top)=\bot \textrm{ .}
\end{equation}

\begin{remark}\label{A(J) and J(A) classically}
Classically, $\AA(\J)$ and $\JJ(\A)$ always exist. In fact
$\AA(\J)=-\J-$ because $-\J-$ is a saturation and, for any other
saturation $\A$, $\A$ is compatible with $\J$ exactly when
$\A\sub-\J-$.\footnote{Intuitionistically, $-\J-$ is indeed a
saturation, but it is not compatible with $\J$ in general. In fact,
for $\J$ $=$ $id$, this would give $--$ as the corresponding saturation; now if $--$
were compatible with $id$, then $a\in--U$, that is $--U\overlap
id\{a\}$, would give $U\overlap id\{a\}$, that is $a\in U$.} Dually,
$\JJ(\A)=-\A-$ because $-\A-$ is a
reduction\footnote{Intuitionistically, $-\A-$ is not even
contractive in general (think of the case $\A=id$).} and, for any
other reduction $\J$, $\J$ is compatible with $\A$ exactly when
$\J\sub-\A-$ (the latter condition is another classical equivalent
to compatibility).\end{remark}

We are going to show that $\AA(\J)$ and $\JJ(\A)$ always exist also
in an intuitionistic, though impredicative framework. Moreover, they
can be constructed also predicatively in many interesting cases (see
section 5).

\subsubsection{The construction of $\AA(\J)$}

Let $\A$ be a saturation and $\J$ be a reduction. From the
equivalence (\ref{eq. semi-Galois connection}), we know that $\A$ is
compatible with $\J$ if and only if $\A$ $\sub$ $\LL(\J)$. Hence, in
order to prove that $\AA(\J)$ exists it is sufficient to check that
$\LL(\J)$ is a saturation.

\begin{lemma}
For every $\O$, the operator $\LL(\O)$ is a saturation.
\end{lemma}
\begin{proof}
We use the characterization of $\LL(\O)$ provided by (\ref{eq def L(O)}).
$\LL(\O)$ is expansive: if $a\in U$, then $a\in\O V$ implies $U\overlap\O V$ for all $V$; so $a\in\LL(\O)U$. $\LL(\O)$ is monotone: if $U\sub
U'$, then $U\overlap\O V$ yields $U'\overlap\O V$; so $\LL(\O)
U\sub\LL(\O) U'$. Since $\LL(\O)$ is expansive, to prove that
$\LL(\O)$ is idempotent it is sufficient to show that
$a\in\LL(\O)\LL(\O) U$ implies $a\in\LL(\O) U$. So we assume
$a\in\LL(\O)\LL(\O) U$ and $a\in\O V$ and we claim $U\overlap\O V$.
The assumptions give  $\LL(\O)
U\overlap\O V$ and hence $U\overlap\O V$  because
$\LL(\O)\comp\O$.
\end{proof}

\begin{corollary}
For every reduction $\J$, the saturation $\AA(\J)$ exists and it is $\AA(\J)=\LL(\J)$, that is:
\begin{equation}\label{eq def A(J)}
a\in\AA(\J)\,U\quad\Longleftrightarrow\quad(\forall\, V\sub
S)(a\in\J V\ \Rightarrow\ U\overlap\J V)\end{equation}(for all $a\in
S$ and $U\sub S$).
 Moreover:
\begin{equation}\label{eq. A-semi-Galois}\A\comp\J\qquad\textrm{if and only if}\qquad\A\sub\AA(\J)\,
.\end{equation}
\end{corollary}
Equivalence (\ref{eq def A(J)}) is nothing but the usual
definition of a closure operator associated with an interior
operator. In fact, it says that a point lies in the closure of
a subset if and only if all its open neighbourhoods intersect that subset.

%
%

\subsubsection{The construction of $\JJ(\A)$}

In the case of $\JJ(\A)$ the situation is somewhat different. In
fact, the constant operator $\RR(\A)$ is not a reduction since it is not
contractive in general, though it is monotone and
idempotent.\footnote{The equation $\JJ(\A)=\RR(\A)$ holds in the
case (and, classically, only in the case) $\A=\bot$.} For instance, $\JJ(id)=id$ while $\RR(id)=\top$. We can
nevertheless prove the following:

\begin{proposition}\label{prop. existence of JJ(A)}
The reduction $\JJ(\A)$ exists for every $\A\in SAT(S)$.
\end{proposition}
\begin{proof}
Put $\JJ(\A)$ $\stackrel{def}{=}$ $\bigvee\{\J\in RED(S)$ $|$
$\A\comp\J\}$ and apply lemma \ref{lemma comp. and unions} and
proposition \ref{prop infAi and supJi}.
\end{proof}

The explicit construction of $\JJ(\A)$ that  we are going to
present has been inspired by the results in \cite{mlsa} (see also section \ref{section predicative}).
We are going to characterize $\JJ(\A)(V)$ as the largest subset
of $V$ that splits $\A$, according to
definition \ref{def. splitting subset}.

\begin{proposition}\label{caratterizzazione di J(A)}
For every saturation $\A$, the reduction $\JJ(\A)$ satisfies:
\begin{equation}\label{eq def J(A)}
a\in\JJ(\A)\,V\quad\Longleftrightarrow\quad(\exists\, Z\sub S)(a\in
Z\sub V\ \land\ Z\textrm{ splits }\A)\end{equation}(for all $a\in S$
and $V\sub S$). Moreover:
\begin{equation}\label{eq. J-semi-Galois}\A\comp\J\qquad\textrm{if and only
if}\qquad\J\sub\JJ(\A)\, .\end{equation}
\end{proposition}
\begin{proof}
Let $\O$ be the operator defined by the right-hand side of (\ref{eq
def J(A)}); so $\O V$ = $\bigcup\{Z\sub V\ |\ Z\
\textrm{splits}\ \A\}$. This shows that $\O$ is a reduction, namely
$\J_P$ of equation (\ref{eq def A and J from a family}) with respect
to the family $P$ = $\{Z\ |\ Z$ splits $\A\}$. Moreover, $\JJ(\A)V$
= $\bigcup\{\J V$ $|$ $\J\in RED(S)$ and $\A\comp\J\}$ $\sub$
$\bigcup\{Z$ $|$ $Z\sub V$ and $Z$ splits $\A\}$ = $\O(V)$ because $\J V\sub V$ and $\J V$ splits $\A$ for all $\J$ such that $\A\comp\J$.
Since by definition  $\JJ(\A)$ is the greatest reduction which is compatible with $\A$,
to prove the opposite inclusion $\O \sub \JJ(\A)$, it is sufficient to show that $\A\comp\O$.
So let $\A U\overlap\O V$; this means that $a\in Z\sub V$ for some
$a\in\A U$ and some $Z$ that splits $\A$. In particular, $\A
U\overlap Z$ and hence $U\overlap Z$; a fortiori, $U\overlap V$ as
wished.

Equation (\ref{eq. J-semi-Galois}) follows from the
definition of $\JJ(\A)$ and from item 2 of lemma \ref{lemma
Trentinaglia}.
\end{proof}

It is easy to check that $Fix(\JJ(\A))$ = $\{Z$ $|$
$Z$ splits $\A\}$, that is, the subsets that split $\A$ are
precisely the fixed points of $\JJ(\A)$.

%

\subsection{The Galois connection}

An immediate consequence of the equivalences (\ref{eq. A-semi-Galois}) and (\ref{eq.
J-semi-Galois}) is the following:

\begin{proposition}\label{prop Galois connection} For every set $S$ and
for all $\A\in SAT(S)$ and $\J\in RED(S)$, the following holds:
\begin{equation}\label{eq. Galois connection}
\A\sub\AA(\J)\quad\Longleftrightarrow\quad\A\comp\J\quad
\Longleftrightarrow\quad\J\sub\JJ(\A)\, .
\end{equation} Therefore the two maps
$\AA : RED(S) \rightarrow SAT(S)$ and $\JJ : SAT(S) \rightarrow
RED(S)$ form an (antitone) Galois connection \cite{birkhoff,mac}.
\end{proposition}

As with any Galois connection, we obtain:
\begin{corollary}\label{corollary properties Galois} The maps $\AA$
and $\JJ$ satisfy:
\begin{enumerate}
\item $\AA$ and $\JJ$ are antitone (that is, order-reversing);
\item $\A\sub\AA\JJ(\A)$ and $\J\sub\JJ\AA(\J)$;
\item $\AA\JJ\AA=\AA$ and $\JJ\AA\JJ=\JJ$;
\item $\AA(\bigcup_i\J_i)=\bigcap_i\AA(\J_i)$ and
$\JJ(\bigcup_i\A_i)=\bigcap_i\JJ(\A_i)$.
\end{enumerate}
A consequence of 3  is that $\A=\AA(\J)$ for some $\J$ if and only if $\AA\JJ(\A)=\A$.
Dually, $\J=\JJ(\A)$ for some $\A$ if and only if $\JJ\AA(\J)=\J$.
\end{corollary}


By remark \ref{A(J) and J(A) classically}, reasoning classically $\AA\JJ$ and $\JJ\AA$ become
the identity on $RED(S)$ and $SAT(S)$, respectively. So every
saturation $\A$ is of the form $\AA(\J)$ and every reduction $\J$ is
of the form $\JJ(\A)$. This is not provable intuitionistically (see
section \ref{counterexamples}).

Some other consequences of the Galois connection will be studied in the following sections.

\section{Basic topologies}\label{section basic topologies}

A set $S$ equipped with a saturation $\A$ and with a reduction $\J$
such that $\A\comp\J$ can be seen as a generalized topological
space. Following \cite{some points,bp}, we put:

\begin{definition}
A \emph{basic topology} is a triple $(S,\A,\J)$ where $S$ is a set,
$\A\in SAT(S)$, $\J\in RED(S)$ and $\A\comp\J$.
\end{definition}

A basic topology generalizes a topological space not only because
$\A$ and $\J$ are not required to preserve finite joins and meets,
respectively, but also because $\A$ could be smaller than that
determined by $\J$, that is, $\AA(\J)$; and dually for $\J$.

In this section, we fix a set $S$ and we consider the collection
$\mathbf{BTop}(S)$ of all basic topologies on $S$. When $S$ is fixed, we write
 $[\A,\J]$ for the basic topology $(S,\A,\J)$. The following
definition makes $\mathbf{BTop}(S)$ a partial order.

\begin{definition}
Let $[\A_1,\J_1]$ and $[\A_2,\J_2]$ be two basic topologies on a set
$S$. We say that $[\A_1,\J_1]$ is \emph{coarser} than $[\A_2,\J_2]$
(equivalently, $[\A_2,\J_2]$ is \emph{finer} than $[\A_1,\J_1]$), and write
$[\A_1,\J_1]\leq[\A_2,\J_2]$, if
$\A_2\sub\A_1$ and $\J_1\sub\J_2$.
\end{definition}

The terms coarser and finer are imported from general topology. They
are justified by the fact that $\J_1\sub\J_2$ and $\A_2\sub\A_1$
mean precisely that $Fix(\J_1)\sub Fix(\J_2)$ and $Fix(\A_1)\sub
Fix(\A_2)$. With respect to this partial order, $\mathbf{BTop}(S)$
becomes a suplattice where the join of a family $[\A_i,\J_i]$ is the
basic topology $[\bigwedge_i\A_i,\bigvee_i\J_i]$.
This is indeed a basic topology by proposition~\ref{prop infAi and supJi}
and the fact that
$\bigwedge_i\A_i\comp\bigvee_i\J_i$, which is proved as follows.
If $\bigcap_i\A_i
U\overlap\bigcup_i\J_i V$, then there exists $k$ such that
$\bigcap_i\A_i U\overlap\J_k V$. So $\A_k U\overlap\J_k V$ and hence
$U\overlap\J_k V$ because $\A_k\comp\J_k$; thus
$U\overlap\bigcup_i\J_i V$.

\subsection{Reduced and saturated basic topologies}

The suplattices $(RED(S),\sub)$ and $(SAT(S),\supseteq)$
can be embedded canonically in $(\mathbf{BTop}(S),\leq)$ by
identifying $\J$ with $[\AA(\J),\J]$ and $\A$ with $[\A,\JJ(\A)]$.
This motivates the following:

\begin{definition}
We call \emph{reduced} a basic topology of the form $[\AA(\J),\J]$
and \emph{saturated} one of the form $[\A,\JJ(\A)]$.
\end{definition}

By (\ref{eq. AA casi banali}), the basic topologies $[id,id]$ and $[\top,\bot]$ are both reduced
and saturated at the same time. As a consequence, $[id,\bot]$ is an
example of a basic topology which is neither reduced nor saturated.
By the way, $[id,\bot]$ is also a counterexample to the implication
$\A\comp\J\ \Rightarrow\ \AA(\J)\comp\JJ(\A)$.

>From a classical point of view, by remark \ref{A(J) and J(A) classically}, a basic topology is reduced if and
only if it is saturated; moreover, the following identities hold:
$$[\AA(\J),\J]=[\AA(\J),\JJ\AA(\J)]\qquad\textrm{and}\qquad
[\A,\JJ(\A)]=[\AA\JJ(\A),\JJ(\A)]$$ for all $\J$ and $\A$. In other words, each reduction
$\J$ represents the same basic topology as its corresponding saturation $\AA(\J)$.
Similarly, for every saturation $\A$,   $\A$ and $\JJ(\A)$
correspond to the same  basic topology.
We shall see in section
\ref{counterexamples} that all this no longer holds intuitionistically.

The identity $\AA\JJ\AA=\AA$  says that $\AA(\J)$ and $\JJ\AA(\J)$
give rise to the same basic topology.
Similarly, $\JJ(\A)$ and $\AA\JJ(\A)$ correspond to the same basic topology because
$\JJ\AA\JJ=\JJ$.
 Finally, provided that $\J$ and $\A$ are identified with the corresponding basic topologies,
 $\J\leq\A$ means precisely that  $\A\comp\J$.
From this perspective, lemma \ref{lemma Trentinaglia} follows from transitivity of $\leq$.

\subsubsection{A decomposition of the Galois connection}

Let $I_R$ and $I_S$ be the functors  embedding
$(RED(S),\sub)$ and $(SAT(S),\supseteq)$, respectively,
into $\mathbf{BTop}(S)$.
So $I_R(\J)=[\AA(\J),\J]$ and $I_S(\A)=[\A,\JJ(\A)]$. In
the opposite direction, we consider a ``forgetful'' map
from $\mathbf{BTop}(S)$ to $RED(S)$ which sends $[\A,\J]$ to $\J$
and ``forgets'' $\A$; similarly for $SAT(S)$. So we put
$U_R([\A,\J])=\J$ and $U_S([\A,\J])=\A$.
Then 
$U_R:(\mathbf{BTop}(S),\leq)\rightarrow(RED(S),\sub)$ and
$U_S:(\mathbf{BTop}(S),\leq)\rightarrow(SAT(S),\supseteq)$
are trivially monotone, and hence functors.

\begin{proposition}\label{aggiunzione UI}
The functor $U_R$ is right adjoint to the embedding $I_R$.
Dually,  $U_S$ is left adjoint to $I_S$. In symbols,
$I_R\dashv U_R$ and $U_S\dashv I_S$.
$$\xymatrix@=5pt{
\big(RED(S),\sub\big)\quad \ar@{->}[rrr]<1ex>^{I_R} & & & \quad
\mathbf{BTop}(S) \quad \ar@{->}[lll]<1ex>^{U_R}
\ar@{->}[rrr]<1ex>^{U_S} & & & \quad \big(SAT(S),\supseteq\big)
\ar@{->}[lll]<1ex>^{I_S} }$$
\end{proposition}
\begin{proof}
We must prove that
$$[\AA(\J),\J]\leq[\A',\J']\Longleftrightarrow\J\sub\J'\quad\textrm{ and }
\quad\A'\supseteq\A\Longleftrightarrow[\A',\J']\leq[\A,\JJ(\A)]$$
for all $\J\in RED(S)$, $\A\in SAT(S)$ and
$[\A',\J']\in\mathbf{BTop}(S)$. We check the latter; the former has
a dual proof.
From $\A'\supseteq\A$ one has $\JJ(\A')\sub\JJ(\A)$ because $\JJ$ is antitone;
from $\A'\comp\J'$ one has $\J'\sub\JJ(\A')$  and hence $\J'\sub\JJ(\A)$.
Together with $\A'\supseteq\A$, this gives
the claim $[\A',\J']\leq[\A,\JJ(\A)]$.
The    other direction is trivial.
\end{proof}

The composition of the two adjunctions gives: $U_S\,I_R\dashv
U_R\,I_S$ between $(RED(S),\sub)$ and $(SAT(S),\supseteq)$. By
unfolding definitions, one sees that this is nothing but the Galois
connection between $\AA$ and $\JJ$.

\subsubsection{Reduction and saturation of a basic topology}

Let us consider the following monotone maps on
$\mathbf{BTop}(S)$: $$(\ )^R\quad\stackrel{def}{=}\quad I_R\,
U_R\quad\textrm{ and }\quad(\ )^S\quad\stackrel{def}{=}\quad I_S\,
U_S\ .$$ By unfolding definitions, one gets:
$$[\A,\J]^R=[\AA(\J),\J]\quad\textrm{ and
}\quad[\A,\J]^S=[\A,\JJ(\A)]$$ for every basic topology $[\A,\J]$.
We call these ``the reduction'' and ``the saturation'' of the basic
topology $[\A,\J]$.
The following is a standard consequence of the adjunctions in
proposition~\ref{aggiunzione UI}.

\begin{corollary}
The endofunctors $(\ )^R$ and $(\ )^S$ are, respectively, a
reduction (comonad) and a saturation (monad) on the poset
$\mathbf{BTop}(S)$.
\end{corollary}

In particular, for every basic topology $[\A,\J]$ one has:
$$[\A,\J]^R=[\AA(\J),\J]\ \leq\ [\A,\J]\ \leq\ [\A,\JJ(\A)]=[\A,\J]^S\ .$$ If
$T=[\A,\J]$, then $T^R$ is the greatest reduced basic topology below
$T$, while $T^S$ is the least saturated basic topology above $T$.
Clearly, $T$ is reduced if and only if $T$ $=$ $T^R$; similarly, $T$
is saturated if and only if $T$ $=$ $T^S$.
%

The following picture presents the general form of the lattice
freely generated by a basic topology $T$ with respect to the operations
$(\ )^R$ and $(\ )^S$.
We here write $\cong$ between two objects which are equal (when they are seen) as basic topologies.

$$
\xymatrix@=5pt{
 & {T^S=[\A,\JJ(\A)]\cong\A} & \\
 & & {T^{SR}=[\AA\JJ(\A),\JJ(\A)]\cong\JJ(\A)\cong\AA\JJ(\A)} \ar@{-}[ul] \\
{T=[\A,\J]} \ar@{-}[uur] & & \\
 & & {T^{RS}=[\AA(\J),\JJ\AA(\J)]\cong\AA(\J)\cong\JJ\AA(\J)} \ar@{-}[uu] \\
 & \ar@{-}[uul] {T^R=[\AA(\J),\J]\cong\J} \ar@{-}[ur] &}
$$

In fact, $T^S$, $T^{SR}$ and $T^{RS}$ are always saturated (because
$\JJ$ = $\JJ\AA\JJ$), so that any further application of $(\ )^S$ on
them gives no new result. Dually, $T^R$, $T^{SR}$ and $T^{RS}$ are
kept fixed by $(\ )^R$ since they are all reduced. All inclusions
are fairly obvious. For instance, to prove $T^{RS}\leq T^{SR}$,
start from $\A\sub\AA(\J)$ and $\J\sub\JJ(\A)$ (compatibility in
$T$); then apply $\JJ$ and $\AA$, respectively, to get (by item 1 of
corollary \ref{corollary properties Galois}) $\JJ\AA(\J)\sub\JJ(\A)$
and $\AA\JJ(\A)\sub\AA(\J)$.

If $T$ is reduced (that is $\A$ = $\AA(\J)$), not only $T=T^R$, but
also $T^{RS}$ $=$ $T^S$ and hence $T^{RS}$ $=$ $T^{SR}$ = $T^S$.
Therefore when $T$ is reduced the picture above collapses to $T\leq
T^S$. Dually, if $T$ is saturated, then $T=T^S$ and $T^R$ $=$
$T^{RS}$ $=$ $T^{SR}$. Hence when $T$ is saturated the picture
becomes $T^R\leq T$.

>From a classical point of view one has (see remark \ref{A(J) and
J(A) classically}): $T^R$ $=$ $[-\J-,\J]$ and $T^S$ $=$ $[\A,-\A-]$.
So the picture above simplifies to $T^{RS}$ $=$ $T^R$
$\leq$ $T$ $\leq$ $T^S$ $=$ $T^{SR}$.

The basic topology $T$ $=$ $[id,\bot]$ provides a counterexample to
all of the following equations: $T^R$ $=$ $T$, $T$ $=$ $T^S$ and
$T^{RS}$ $=$ $T^{SR}$. In fact, thanks to equations (\ref{eq. AA
casi banali}) we get $T^R$ $=$ $[\top,\bot]$ $=$ $T^{RS}$ and $T^S$
$=$ $[id,id]$ $=$ $T^{SR}$. In section~\ref{counterexamples} we will
give counterexamples to the (classically valid) equations $T^R$
$=$ $T^{RS}$ (``every reduced basic topology is saturated'') and
$T^S$ $=$ $T^{SR}$ (``every saturated basic topology is reduced'').

\section{Some counterexamples}\label{counterexamples}

Contrary to what happens classically, we are going to show that the classes of reduced and of saturated basic topologies are not equal intuitionistically.
Actually, neither of the two classes contains the other. We begin by showing 
several equivalent manifestations of the two inclusions.

\begin{proposition}\label{prop. equivalenti controesempi} The following are equivalent:
\begin{enumerate}
\item every saturated basic topology is reduced;
\item $T^{SR}=T^S$, for every basic topology $T$;
\item $\AA\JJ(\A)=\A$, for every $\A\in SAT(S)$;
\item $\AA$ is surjective (every $\A$ is of the form $\AA(\J)$ for some $\J$);
\item $\JJ$ is injective ($\JJ(\A)=\JJ(\A')$ only if $\A=\A'$).
\end{enumerate} Dually, also the following are equivalent:
\begin{enumerate}
\item every reduced basic topology is saturated;
\item $T^{RS}=T^R$, for every basic topology $T$;
\item $\JJ\AA(\J)=\J$, for every $\J\in RED(S)$;
\item $\JJ$ is surjective (every $\J$ is of the form $\JJ(\A)$ for some $\A$);
\item $\AA$ is injective ($\AA(\J)=\AA(\J')$ only if $\J=\J'$).
\end{enumerate}
\end{proposition}
\begin{proof}
We prove only the first half of the statement, since the other half
is dual. (1$\Leftrightarrow$2$\Leftrightarrow$3): a basic topology
is saturated iff it is of the form $T^S$ $=$ $[\A,\JJ(\A)]$ for some
$T$; so 1 holds iff every $T^S$ is reduced iff every $T^S$
coincides with its reduction $T^{SR}$, that is 2, which means that
every $[\A,\JJ(\A)]$ coincides with $[\AA\JJ(\A),\JJ(\A)]$, which is
equivalent to 3. (3$\Leftrightarrow$4$\Leftrightarrow$5): this holds
for every Galois connection, since it follows from corollary
\ref{corollary properties Galois}.
\end{proof}

\subsection{Not every closure is determined by an interior}

We show that not every $\A$ is of the form $\AA(\J)$ for some $\J$, that is item 4 
of the first part of proposition~\ref{prop. equivalenti controesempi}.
We actually give a counterexample for its equivalent formulation in item 3.

\begin{lemma}\label{density gives Pos}
For every set $S$ and every $\J\in RED(S)$, the following holds
\begin{equation}\label{eq positivity axiom}
\big(a\in\J S\Rightarrow a\in\AA(\J)U\big)\Longrightarrow a\in\AA(\J)U
\end{equation}
 for all $a\in S$ and $U\sub U$.
\end{lemma}
\begin{proof}
Assume $a\in\J S$ $\Rightarrow$
$a\in\AA(\J) U$. The claim is $a\in\J V$ $\Rightarrow$ $U\overlap\J
V$ for all $V\sub S$. 
So let $a\in\J V$. Then we immediately have $a\in\J S$ and hence $a\in\AA(\J)U$ by the 
assumption. By definition of $\AA(\J)$ we obtain that $a\in\J V$
$\Rightarrow$ $U\overlap\J V$; hence the claim $U\overlap\J V$
because $a\in\J V$.
\end{proof}

\begin {proposition}\label{AJ identity = LEM}
The fact  that $\AA\JJ$ is the identity on $SAT(S)$ for every $S$ is equivalent to the law of excluded middle.
\end{proposition}
\begin{proof}
Fix an arbitrary proposition $p$, let $S=\{*\}$ (the one-element set) and put $\A U\ =\ U\cup\{x\in S\ |\ p\}$ for every $U\sub S$.  
This obviously defines a saturation. We claim that the assumption $\A=\AA\JJ(\A)$ yields $\neg\neg p\rightarrow p$. 
Assuming $\A=\AA\JJ(\A)$,  the previous lemma would give in particular
  $(*\in\JJ(\A) S\Rightarrow *\in\A\emptyset)$ $\Rightarrow$ $*\in\A\emptyset$ .  
  By definition, $*\in\A\emptyset$ is equivalent to $p$.
On the other hand,  $*\in\JJ(\A)S$ is $\exists\,Z(*\in Z \sub S\land Z\textrm{ splits }\A)$, that is, $\{*\}$ splits $\A$. 
This means that $*\in\A U\Rightarrow *\in U$ for all $U\sub S$; in other words, it says that $\A U\sub U$ for all $U$. By the definition of $\A$, this is equivalent to $\{x\in S\ |\ p\}\sub U$ for all $U$ and hence to $\{x\in S\ |\ p\}\sub\emptyset$, that is, $\neg p$. So $(*\in\JJ(\A) S\Rightarrow *\in\A\emptyset)\Rightarrow *\in\A\emptyset$ is tantamount to $(\neg p\rightarrow p)\rightarrow p$ which is in turn equivalent to $\neg\neg p\rightarrow p$, since $(\neg p \rightarrow p) \leftrightarrow \neg\neg p$.
\end{proof}

An alternative argument to show that not every $\A$ is of the form $\AA(\J)$
uses a result in  \cite{CSSS}.
There the authors construct a class of saturations  and
show (corollary 3) that Markov's principle follows from the hypothesis that all
such saturations admit a \emph{positivity predicate}.\footnote{This is linked to  the
well-known fact that intuitionistically not all locales are open.} 
If every $\A$ were of the form $\AA(\J)$ for some $\J$, then by  lemma~\ref{density gives Pos} we would obtain
  an expression, in the present framework, of the fact that $\A$ admits a positivity predicate.

\subsection{Not every interior is determined by a closure}


We show that not every $\J$ is of the form $\JJ(\A)$ for some $\A$  
by giving a counterexample for its equivalent formulation $\J=\JJ\AA(\J)$.

Recall that $\J=\JJ\AA(\J)$ holds iff $\JJ\AA(\J)\sub\J$ iff
$Fix(\JJ\AA(\J))\sub Fix(\J)$ iff, by the remark after
proposition~\ref{caratterizzazione di J(A)},
 $\{Z\sub S$ $|$ $Z$ splits $\AA(\J)\}\sub Fix(\J)$. Therefore, to show that the identity $\J=\JJ\AA(\J)$ cannot hold
intuitionistically for all $\J$, it is sufficient to find a set $S$, a reduction
$\J$ on $S$ and a subset $Z\sub S$ which splits $\AA(\J)$ and such
that the equality $Z=\J Z$ cannot hold intuitionistically. In fact, we will show that  if the identity under consideration were true, then the law of excluded middle would hold.

Let $S=\{a,b\}$ with $a\neq b$. For every proposition $p$, we
consider the map $\J_p:Pow(S)\rightarrow Pow(S)$ defined by: \begin{equation}\label{eq def
J_p}x\in\J_pU\ \stackrel{def}{\Longleftrightarrow}\ x\in U\ \land\
\big(b\notin U\Rightarrow p\big)\end{equation} for every
$U\sub S$ and $x\in S$.
In particular, by intuitionistic logic, one gets $b\in\J_pU$ iff $b\in U$.

\begin{lemma}
For every proposition $p$, the map $\J_p$ is a reduction on $\{a,b\}$.
\end{lemma}
\begin{proof}
The inclusion $\J_p U\sub U$ holds trivially for all $U\sub S$.
Given this, it is sufficient to check that $\J_p U\sub V$ implies
$\J_p U\sub\J_p V$ for all $U,V\sub S$. So we assume $x\in\J_p U\sub
V$, for $x\in S$, and we show $x\in\J_p V$, that is, $x\in V$ and 
$b\notin V\Rightarrow p$. The former follows easily from 
$x\in\J_p U\sub V$. To prove the latter, first note that
$b\notin V$ implies $b\notin U$. In fact, if it were $b\in U$, then it would also be $b\in\J_p U$ 
and hence $b\in V$ by the hypothesis $\J_pU\sub
V$; a contradiction. Therefore $b\notin U\Rightarrow p$ yields
$b\notin V\Rightarrow p$. But $b\notin U\Rightarrow p$ is part
of the hypothesis $x\in\J_pU$. This completes the proof.
\end{proof}

Now we choose the subset $Z = \{a\}$ and show that $Z$ splits
$\AA(\J_p)$ but cannot be proven to equal $\J_pZ$.  
\begin{lemma}
$\{a\}=\J_p\{a\}$ holds if and only if $p$ is true.
\end{lemma}
\begin{proof}
$\{a\}=\J_p\{a\}$ iff $\{a\}\sub\J_p\{a\}$ iff $a\in\J_p\{a\}$
which, by definition, means
$a\in\{a\}\land (b\notin\{a\}\Rightarrow p)$ which  is
equivalent to $p$ since $b\neq a$.
\end{proof}

>From now on we restrict $p$ to be a proposition such that $\neg\neg p$ holds. For example, one can choose $p$ of the form $\varphi\vee\neg\varphi$. 

\begin{lemma}
 The subset $\{a\}$ splits $\AA(\J_p)$ for every $p$ such that $\neg\neg p$ holds.
\end{lemma}
\begin{proof}
Recall that $\{a\}$ splits $\AA(\J_p)$ if
$(\forall\,U\sub S)(\AA(\J_p)U\overlap\{a\}$ $\Rightarrow$ $U\overlap
\{a\})$, that is, $(\forall\,U\sub S)(a\in\AA(\J_p)U\Rightarrow a\in
U)$. Let $U\sub S$ and $a\in\AA(\J_p)U$; our claim is $a\in U$. 
By (\ref{eq def A(J)}), $a\in\AA(\J_p)U$ means that $a\in\J_p V$
$\Rightarrow$ $U\overlap\J_pV$ for all $V\sub S$. By specializing to the
case $V=\{a\}\cup V_b$ where $V_b\stackrel{def}{=}\{x\in\{b\}\ |\
a\in U\}$, we get:
\begin{equation}\label{eq finestra}
a\in\J_p(\{a\}\cup V_b)\quad\Longrightarrow\quad U\overlap
\J_p(\{a\}\cup V_b)\,.
\end{equation} 
Note that $b\notin \{a\}\cup V_b\Rightarrow p$ iff 
$b\notin V_b\Rightarrow p$ iff $a\notin U\Rightarrow p$.
Therefore the antecedent of (\ref{eq finestra}), that is 
$a\in\{a\}\cup V_b$ $\land$ $(b\notin \{a\}\cup V_b\Rightarrow p)$,  
is equivalent to $a\notin U\Rightarrow p$. 
On the other hand, the consequent of (\ref{eq finestra}) implies that 
$U\overlap\{a\}\cup V_b$ and so $U\overlap\{a\}$ or $U\overlap V_b$. 
In either case, one can derive $a\in U$. Therefore
(\ref{eq finestra}) yields
\begin{equation}\label{eq finestra bis}
(a\notin U\Rightarrow p)\Rightarrow a\in
U\,.
\end{equation} Recall that our aim is to prove $a\in U$. Assume $a\notin U$. Then (\ref{eq
finestra bis}) becomes $p\Rightarrow a\in U$. Together with $a\notin U$, this gives $\neg p$, thus contradicting the fact that
 $\neg\neg p$ holds. So $a\notin U$ must be false; hence the antecedent of (\ref{eq finestra bis}) becomes
true and we are done.
\end{proof}

By putting all lemmas together, we have

\begin{proposition}\label{JA identity = LEM}
The fact that $\JJ\AA$ is the identity on $RED(S)$ for every $S$ is equivalent to the law of excluded middle.
\end{proposition}
\begin{proof}
Assume $\neg\neg p$. Chose $S=\{a,b\}$ and construct $\J_p$ as above. Then $\{a\}$ splits $\AA(\J_p)$, that is, $\{a\}\in Fix(\JJ\AA(\J_p))$. If $\JJ\AA$ were the identity, then $\{a\}=\J_p\{a\}$. So $p$ would be true.
\end{proof}

In \cite{grayson}, Grayson  shows that there exists a model
for intuitionistic analysis in which
 real numbers can be equipped with two different topologies, hence two different reductions, $\J_1$ and $\J_2$
say, which are
associated with the same closure operator (saturation).
Since the notion of closure associated with $\J$ in \cite{grayson} is
precisely our $\AA(\J)$,  one thus has
$\AA(\J_1)$ $=$ $\AA(\J_2)$, even if $\J_1 = \J_2$ does not hold. So $\AA$ cannot be proven to be
injective.

A third argument makes use of two  notions of sobriety for a topological space.
One can show (see  \cite{bp}) that if every reduced basic topology is
saturated, then the two notions  coincide, while this does not hold constructively (see \cite{aczel-fox,fourman}).

\begin{remark}
As a consequence of propositions~\ref{AJ identity = LEM} and \ref{JA identity = LEM}, 
all conditions in both parts of proposition~\ref{prop. equivalenti controesempi} are equivalent to the law of excluded middle, and hence they are also equivalent one another.
\end{remark}

\section{Generated and representable basic topologies}\label{section predicative}

In this final section we  show that the functors $\AA$ and $\JJ$ can be defined predicatively
on a wide class of reductions and saturations, respectively. Impredicatively, such classes coincide with the class of all reductions and of all saturations.

\subsection{Representable basic topologies}

For every binary relation $r$ between two sets $X$ and $S$, the operators $r: Pow(X) $ $ \rightarrow $ $ Pow(S)$ of \emph{direct image} and $r^-:
Pow(S)\rightarrow Pow(X)$ of \emph{inverse
image} are defined by $$
r D\stackrel{def}{=}\{a\in S\ |\
(\exists\,x\in D)(x\,r\,a)\}\ \textrm{ and }\ r^-
U\stackrel{def}{=}\{x\in X\ |\ (\exists\,a\in U)(x\,r\,a)\}
$$ 
for
all $D\sub X$ and $U\sub S$. Since $r$ and $r^-$ preserve unions,
they admit right adjoints given by
$$r^* U\stackrel{def}{=}\{x\in
X\ |\ r\{x\}\sub U\}\textrm{ and }r^{-*} D\stackrel{def}{=}\{a\in S\
|\ r^-\{a\}\sub D\}\ .$$ The fact that the operators $r$ and $r^-$
come from the same relation is expressed ``algebraically'' by:
\begin{equation}\label{eq. fund. symmetry}
r D\overlap U\quad\Longleftrightarrow\quad D\overlap
r^-U\end{equation} (for all $D\sub X$ and $U\sub S$). We shall refer
to this condition as $r\fusim r^-$, read ``$r$ and $r^-$ are
symmetric''.\footnote{This notion can be treated algebraically in
the framework of overlap algebras \cite{cisa,ci2,bp}. Classically, it corresponds to
the notion of conjugate functions between complete Boolean algebras
as studied by J\'onsson and Tarski \cite{tarski}.}

\begin{proposition}\label{prop. operators image}
For every binary relation $r$ between two sets $X$ and $S$, the
structure $(S,r^{-*}r^-,rr^*)$ is a basic topology.
\end{proposition}
\begin{proof} Since $r\dashv r^*$ and $r^-\dashv r^{-*}$, it follows
that $rr^*$ is a reduction and $r^{-*}r^-$ is a saturation. It
remains to be checked that $r^{-*}r^-\comp rr^*$. If $r^{-*}r^-
U\overlap rr^* V$, then $r^-r^{-*}r^- U\overlap r^* V$ (because
$r\fusim r^-$) and therefore $r^- U\overlap r^* V$ because $r^-r^{-*}$
is contractive; so $U\overlap rr^* V$ (again because $r\fusim r^-$).
\end{proof}

We call \emph{representable} a basic topology obtained in this way.
Also, we say that a reduction $\J$ is representable if it is of the
form $rr^*$ for some relation $r$. Similarly for a saturation.

\begin{proposition}\label{lemma representable is dense}
For every relation $r$ between two sets $X$ and $S$, the basic
topology $(S,r^{-*}r^-,rr^*)$ is reduced, that is:
$$\AA(rr^*)\quad=\quad r^{-*}r^-\ .$$
\end{proposition}
\begin{proof}
Let $\A$ be any other saturation compatible with $rr^*$; we must
show that $\A$ $\sub$ $r^{-*}r^-$. By $r^-\dashv r^{-*}$, our claim
reduces to $r^-\A$ $\sub$ $r^-$. So let $a\in r^-\A U$, that is, $\A
U\overlap r\{a\}$. A general consequence of the adjunction $r\dashv
r^*$ is that $rr^*r$ $=$ $r$ (triangular equality, see \cite{mac});
hence $r\{a\}\in Fix(rr^*)$. Thus we can apply compatibility between
$\A$ and $rr^*$ to $\A U\overlap r\{a\}$ and get $U\overlap r\{a\}$,
that is, $a\in r^- U$.
\end{proof}

So $\AA(\J)$ can be constructed predicatively at least when $\J$
is representable. One can show that impredicatively every
reduction is representable. In fact, let $X$ $=$ $Fix(\J)$ and
consider the relation $r$ given by: $\J U\,r\, a$ if $a\in\J U$. So
$r\{\J U\}$ $=$ $\{a\in S$ $|$ $\J U\r a\}$ $=$ $\{a\in S$ $|$
$a\in\J U\}$ $=$ $\J U$. We have: $a\in rr^*U$ iff
$r^{-}\{a\}\overlap r^*U$ iff $(\exists\, \J V\in X)(\J V\in
r^{-}\{a\}$ $\land$ $\J V\in r^*U)$ iff $(\exists\, \J V\in X)(a\in
r\{\J V\}$ $\land$ $r\{\J V\}\sub U)$ iff $(\exists\, \J V\in
X)(a\in\J V$ $\land$ $\J V\sub U)$. Since $\J$ is a reduction, this
is tantamount to $a\in\J U$. 
It is interesting that unfolding the
definition of $r^{-*}r^-$ in this case one obtains precisely the
characterization of $\AA(\J)$ given in (\ref{eq def A(J)}). In fact:
$a\in r^{-*}r^-U$ iff $r^{-}\{a\}\sub r^-U$ iff $(\forall\, \J V\in
X)(\J V\in r^{-}\{a\}$ $\Rightarrow$ $\J V\in r^-U)$ iff $(\forall\,
\J V\in X)(a\in r\{\J V\}$ $\Rightarrow$ $U\overlap r\{\J V\})$ iff
$(\forall\, \J V\in Fix(\J))(a\in\J V$ $\Rightarrow$ $U\overlap\J V)$.

\subsection{Generated basic topologies}

In this section we are going to show that $\JJ(\A)$ can be
constructed predicatively for an important class of saturations,
namely those which can be generated inductively (see \cite{CSSV}).

In \cite{mlsa}  a quite general method is given for
generating basic topologies. One starts from a
family of sets $\{I(a)$ $|$ $a\in S\}$ and subsets $C(a,i)$ $\sub S$
for all $a\in S$ and $i\in I(a)$; this is called an
\emph{axiom-set}. Next one gives rules to generate the least
saturation $\A_{I,C}$ satisfying $a\in\A_{I,C}(C(a,i))$ and, at the
same time, the greatest reduction $\J_{I,C}$ which is compatible
with $\A_{I,C}$, that is, $\JJ(\A_{I,C})$ in our terminology. We are
going to present such rules, though in a slightly different way, and
prove again the main properties of $\A_{I,C}$ and $\J_{I,C}$ given in \cite{mlsa}, 
in particular that $\J_{I,C}$ $=$ $\JJ(\A_{I,C})$.

We say that a subset $P\sub S$
\emph{fulfills the axiom-set} $I,C$ if $C(a,i)\sub P$ $\Rightarrow$ $a\in
P$ for all $a\in S$ and $i\in I(a)$. For every $U\sub S$, let
$\A_{I,C}(U)$ be the subset of $S$ defined by the following
clauses (see \cite{CSSV}): 
\begin{enumerate}
\item $U\sub\A_{I,C}(U)$;
\item $\A_{I,C}(U)$ fulfills the axiom-set $I,C$;
\item if $U\sub P\sub S$ and $P$ fulfills $I,C$, then $\A_{I,C}(U)\sub P$.
\end{enumerate} In other words, $\A_{I,C}(U)$ is the
least subset which contains $U$ and fulfills the axioms. Since subsets
fulfilling the axioms are closed under arbitrary intersections (as it
is easy to check), we can express $\A_{I,C}(U)$ as:
\begin{eqnarray}\label{eq. def. A_R}
\A_{I,C}(U) & = & \bigcap\{P\sub S\ |\ U\sub P\ \land\ P\
\textrm{fulfills}\ I,C\}\textrm{ .}
\end{eqnarray} This shows at once that the operator $\A_{I,C}$ is a saturation,
being of the form (\ref{eq def A and J from a family}) with respect
to the family of all subsets fulfilling the axioms $I,C$.
It also follows  that $Fix(\A_{I,C})$ is exactly
the collection of subsets fulfilling the axioms.

In \cite{mlsa}, the authors propose a dual construction of
$\J_{I,C}(V)$ for $V\sub S$. Contrary to the construction of
$\A_{I,C}U$ which is inductive, the definition of $\J_{I,C}V$ is
coinductive. The rules given in \cite{mlsa} say that $\J_{I,C}(V)$
\emph{ is the greatest subset of $V$ which ``splits''} $I,C$ according to
the following.

\begin{definition}
 We say that a subset $Z\sub S$
\emph{splits the axiom-set} $I,C$ on a set $S$ if $$a\in Z\Longrightarrow
C(a,i)\overlap Z$$ for all $a\in S$ and $i\in I(a)$.
\end{definition}
Since splitting subsets are closed under unions, we get:
\begin{eqnarray}\label{eq. def. J_R}
\J_{I,C}(V) & = & \bigcup\{Z\sub S\ |\ Z\sub V\ \land\ Z\
\textrm{splits}\ I,C\}\textrm{ .}
\end{eqnarray} This shows that $\J_{I,C}$ is a reduction, namely
that which is associated, according to equation (\ref{eq def A and J from a family}), 
with the family of all splitting subsets.

\begin{lemma}
For every axiom-set $I,C$ on $S$ and for every $Z\sub S$ one has:
$$Z\textrm{ splits }I,C\qquad\textrm{ if and only if }\qquad Z=\J_{I,C}(Z)\ .$$
\end{lemma}
\begin{proof} If $Z$ splits $I,C$, then the union of all
splitting subsets contained in $Z$ gives $Z$ itself. Vice versa,
recall that $\J_{I,C}(Z)$ splits $I,C$ by definition.
\end{proof}

\begin{proposition}\label{lemma basic topology generated}
For every axiom-set $I,C$ on a set $S$, the operators $\A_{I,C}$ and
$\J_{I,C}$ are compatible and the basic topology
$(S,\A_{I,C},\J_{I,C})$ is saturated, that is:
$$\JJ(\A_{I,C})\quad=\quad\J_{I,C}\ .$$
\end{proposition}
\begin{proof}
In order to show that $\A_{I,C}$ and $\J_{I,C}$ are compatible, let
$U,V\sub S$ and consider the subset $P$ $=$ $\{a\in S$ $|$
$a\in\J_{I,C}V$ $\Rightarrow$ $U\overlap\J_{I,C}V\}$. Then the
instance of compatibility $\A_{I,C}U\overlap\J_{I,C}V$
$\Rightarrow$ $U\overlap\J_{I,C} V$ is logically equivalent to $\A_{I,C} U\sub
P$. Since clearly $U\sub P$, in order to obtain $\A_{I,C} U\sub
P$ using clause 3 of the definition of $\A_{I,C}U$, we only need
to show that $P$ fulfills $I,C$.
In other words, we
must prove that $U\overlap\J_{I,C} V$ holds under the assumptions
$C(a,i)\sub P$ and $a\in\J_{I,C}V$. From $a\in\J_{I,C}V$ one gets
$C(a,i)\overlap\J_{I,C}V$ because $\J_{I,C}V$ splits $I,C$. Hence also $P\overlap\J_{I,C}V$ since $C(a,i)\sub
P$. So there exists $a'\in P$ such that $a'\in\J_{I,C}V$. By the
definition of $P$, this implies that $U\overlap\J_{I,C}V$.

Finally, let $\J'$ be another reduction which is compatible with
$\A_{I,C}$. Since by definition $\J_{I,C}V$ is, for all $V\sub S$,
the greatest splitting subset contained in $V$, to prove
$\J'\sub\J_{I,C}$ it is sufficient to check that $\J'V$ splits
$I,C$. So let $a\in S$, $i\in I(a)$ and assume $a\in\J'V$. By the
definition of $\A_{I,C}$ we have $a\in\A_{I,C}(C(a,i))$. So
$\A_{I,C}(C(a,i))\overlap\J'V$ and hence, by compatibility of $\J'$ and 
$\A_{I,C}$,  $C(a,i)\overlap\J'V$ as wished.
\end{proof}

In other words, $\JJ(\A)$ admits a predicative construction whenever
$\A$ can be inductively generated. Note that this is always the case 
if one works within an impredicative framework. In
fact, one can take $I(a)$ $=$ $\{U\sub S$ $|$ $a\in\A U\}$ and
$C(a,U)$ $=$ $U$. In this case, a subset $Z$ splits the axioms
precisely when $a\in Z$ $\Rightarrow$ $U\overlap Z$ for every $a$
and $U$ such that $a\in\A U$. This means $\A U\overlap Z$
$\Rightarrow$ $U\overlap Z$ for every $U$, which  says precisely that
$Z$ splits $\A$ according to definition \ref{def. splitting subset}.
Thus, as expected, $\J_{I,C}$ defined by
 (\ref{eq. def. J_R}) coincides with $\JJ(\A)$, as defined by  (\ref{eq def J(A)}).


\begin{thebibliography}{9}

\bibitem{aczel-fox} P. Aczel and C. Fox, \emph{Separation properties in
Constructive Topology}, in \emph{From Sets and Types to Topology and
Analysis: towards Practicable Foundations for Constructive
Mathematics}, L. Crosilla and P. Schuster eds., Oxford Logic Guides 48,
Oxford U. P., 2005, pp 176-192.

\bibitem{birkhoff} G. Birkhoff, \emph{Lattice Theory}, Amer. Math. Soc. Coll. Pub., Vol 25,
1940.

\bibitem{ci2} F. Ciraulo,
\emph{Regular opens in formal topology and a representation theorem
for overlap algebras}, Ann. Pure Appl. Logic, to appear.

\bibitem{cisa} F. Ciraulo and G. Sambin,
\emph{The overlap algebra of regular opens}, J. Pure Appl. Algebra
214 (2010), pp. 1988 -1995.

\bibitem{CSSS} T. Coquand, S. Sadocco, G. Sambin, J. Smith,
\emph{Formal topologies on the set of first-order formulae}, J.
Symbolic Logic 65 (2000), pp. 1183-1192.

\bibitem{CSSV} T. Coquand, G. Sambin, J. Smith, S. Valentini, \emph{Inductively
generated formal topologies}, Ann. Pure Appl. Logic 124 (2003), pp.
71-106.

\bibitem{fourman} M. Fourman and D. Scott, \emph{Sheaves and
logic}, in \emph{Applications of sheaves. Proceedings of the
Research Symposium on Applications of Sheaf Theory to Logic,
Algebra, and Analysis, Durham, July 21, 1977}, M. Fourman, C.
Mulvey and D. Scott, eds., Lect. Notes Math. 753, 1979.

\bibitem{grayson} R. J. Grayson, \emph{On closed subsets of the intuitionistic reals}, Z. Math. Logik
29 (1983), pp. 7-9.

\bibitem{stone spaces} P. T. Johnstone, \emph{Stone spaces}, Cambridge Studies in Advanced
Mathematics 3, Cambridge U. P., New York, 1983.

\bibitem{tarski} B. J\'onsson, A. Tarski, \emph{Boolean algebras with operators. I.}, Am. J. Math.
73 (1951), pp. 891-939.

\bibitem{joyal-tierney} A. Joyal and M. Tierney, \emph{An Extension of the Galois
Theory of Grothendieck}, Memoirs Amer. Math. Soc. 51 (1984).

\bibitem{mac} S. Mac Lane,
\emph{Categories for the Working Mathematician} (Second ed.),
Springer, 1998.


\bibitem{mlsa} P. Martin-L\"of, G. Sambin \emph{Generating positivity by
coinduction}, in \cite{bp}.


\bibitem{some points} G. Sambin, \emph{Some points in formal topology},
Theor. Comput. Sci. 305 (2003), pp. 347-408.

\bibitem{bp}G. Sambin,
\emph{The Basic Picture and Positive Topology: Structures for
Constructive Mathematics}, Oxford University Press, to appear.


\end{thebibliography}
\end{document}